\documentclass[12pt]{amsart}
\usepackage{cmtiup,amsmath}
\usepackage{amssymb}
\usepackage{amsfonts}

\newtheorem{theorem}{Theorem}[section]
\newtheorem{lemma}[theorem]{Lemma}

\newtheorem{corollary}[theorem]{Corollary}
\newtheorem{proposition}[theorem]{Proposition}

\theoremstyle{definition}
\newtheorem{remark}[theorem]{Remark}
\newcommand{\bp}{\begin{proof}}
\newcommand{\ep}{\end{proof}}
\newcommand{\bpr}{\begin{proposition}}
\newcommand{\epr}{\end{proposition}}
\newcommand{\bl}{\begin{lemma}}
\newcommand{\el}{\end{lemma}}
\newcommand{\F}{{\Bbb F}}
\newcommand{\f}{{\varphi }}

\let\leq=\leqslant
\let\geq=\geqslant
\textwidth=16cm \textheight=24cm \numberwithin{equation}{section}
\begin{document}

\title{On finite soluble groups\\ with almost fixed-point-free automorphisms\\ of non-coprime order}

\author{{E. I. Khukhro}}%

\address{Sobolev Institute of Mathematics, Novosibirsk, Russia,\newline  and University of Lincoln, UK}%
\email{khukhro@yahoo.co.uk}%

\thanks{The work is supported by  Russian Science Foundation (project
14-21-00065)}

\begin{abstract} It is proved that if a finite $p$-soluble group $G$ admits an automorphism $\varphi$ of order $p^n$  having at most $m$ fixed points on every $\f$-invariant elementary abelian $p'$-section of $G$, then the $p$-length of $G$ is bounded above in terms of $p^n$ and $m$; if in addition the group $G$ is soluble, then the Fitting height of $G$ is   bounded above in terms of $p^n$ and $m$. It is also proved that if a finite soluble group $G$ admits an automorphism $\psi$ of order $p^aq^b$ for some primes $p,q$, then the Fitting height of $G$ is bounded above in terms of $|\psi |$ and $|C_G(\psi )|$.
\end{abstract}


\maketitle
\centerline{\textit{to Yurii Leonidovich Ershov on the occasion of his 75-th birthday}}
\section{Introduction}

Studying groups with ``almost fixed-point-free'' automorphisms means obtaining restrictions on the structure of  groups depending on their automorphisms and certain restrictions imposed on the fixed-point subgroups. In this paper we consider questions of bounding the $p$-length and Fitting height of finite $p$-soluble and soluble groups admitting almost fixed-point-free automorphisms of non-coprime order.

Let $\f \in {\rm Aut\,}G$ be an automorphism of a finite group $G$. Studying the structure of the group $G$ depending on $\f $ and the  fixed-point subgroup $C_G(\f )$ is one of the most important and fruitful avenues in finite group theory. The celebrated  Brauer--Fowler theorem  \cite{bra-fow55} (bounding the index of the soluble radical in terms of the order of $|C_G(\f )|$ when $|\f |=2$) and Thompson's theorem \cite{tho59} (giving the nilpotency of $G$ when $\f $ is of prime order and acts fixed-point-freely, that is, $C_G(\f)=1$)  lie in the foundations of the classification of finite simple groups. The classification  was used for obtaining further results on solubility of $G$, or of a suitable ``large'' subgroup. For example, using the classification Hartley \cite{har92} generalized the Brauer--Fowler theorem to any order of $\f$:  the group  $G$ has a soluble subgroup of index bounded in terms of $|\f |$ and $|C_G(\f )|$.

Now suppose that the group $G$ is soluble. Further information on the structure of $G$ is sought first of all in the form of bounds for the Fitting height (nilpotent length). A bound for the Fitting height naturally reduces further studies to the case of nilpotent groups with (almost) fixed-point-free automorphisms, for which, in turn,  problems arise of bounding the derived length, or the nilpotency class of the group or  of  a suitable ``large'' subgroup. Such bounds for nilpotent groups so far have been obtained in the cases of $\f$ being of prime order or of order $4$ in \cite{hi,kr, kr-ko,khu90,kov, mak-khu06}. In addition, definitive general results have been obtained in the study of almost fixed-point-free $p$-automorphisms of finite $p$-groups \cite{alp,khu85,sha93,khu93,me,jai}.

On bounding the Fitting height, especially strong results have been obtained in the case of  soluble groups of automorphisms $A\leq  {\rm Aut\,}G$ of coprime order.
    Thompson \cite{tho64} proved that if both groups $G$ and $A$ are soluble and have coprime orders, then the Fitting height of $G$ is bounded in terms of the Fitting height of $C_G(A)$ and the number $\alpha (A)$ of prime factors of $|A|$ with account for multiplicities. Later the bounds in Thompson's theorem were improved in numerous papers, with definitive results obtained by Turull \cite{tur84} and Hartley and Isaacs \cite{har-isa90} %!
    with linear bounds in terms of $\alpha (A)$ for the Fitting height of the group or of a ``large subgroup''.

The case of non-coprime orders of $G$ and $A\leq {\rm Aut\,}G$ is more difficult. Bell and Hartley \cite{be-ha} constructed examples showing that for any non-nilpotent finite group $A$ there are soluble groups $G$  of  arbitrarily high Fitting height admitting $A$ as a fixed-point-free group of automorphisms.   But if $A$ is nilpotent  and $C_G(A)=1$,  then the Fitting height of $G$ is bounded in terms of $\alpha (A)$ by a special case of Dade's  theorem \cite{dad69}. %!
Unlike the aforementioned ``linear'' results in the coprime case, the bound in Dade's theorem is exponential. Improving this bound to a linear one is a difficult problem; it was tackled in some special cases by Ercan and G\"{u}lo\u{g}lu \cite{er12,er04,er08}.

In the almost fixed-point-free situation, even for a cyclic group of automorphisms $\langle \f\rangle\leq {\rm Aut\,}G$ it is still an open problem to obtain a bound for the Fitting height of a finite soluble group $G$ in terms of $|\f |$ and $|C_G(\f )|$ (this question is equivalent to the one recorded by Belyaev in Kourovka Notebook \cite{kour} as Hartley's Problem~13.8(a)). Beyond the fixed-point-free case of Dade's theorem, so far the only cases where an affirmative solution is known are the cases of automorphisms of primary order $p^n$ (Hartley and Turau \cite{ha-tu}) and of biprimary order $p^aq^b$ (which is discussed in the present paper). 
%! An even more difficult problem in the non-coprime case is obtaining a subgroup of index bounded in terms of $|\f |$ and $|C_G(\f )|$ whose Fitting height would be bounded in terms of $\alpha (\langle\f\rangle)$. 

Another generalization of fixed-point-free automorphisms in the non-coprime case is   Thompson's problem on bounding the $p$-length of a finite $p$-soluble group $G$ admitting
a $p$-group of automorphisms $P$ that acts fixed-point-freely on every $P$-invariant $p'$-section of $G$.  Rae \cite{rae} and Hartley and Rae \cite{ha-ra} solved this problem
in the affirmative for $p\ne 2$, as well as for cyclic $P$ for any $p$. A special case of this problem is when a $p$-soluble group $G$ admits a so-called  $p^n$-splitting automorphism $\f$, which means that $xx^{\f }x^{\f ^2}\cdots x^{\f ^{p^n-1}}=1$ for all $x\in G$ (this  also implies $\f^{p^n}=1$); then of course $\f$ automatically acts fixed-point-freely on $\f$-invariant $p'$-sections. This case was actually considered earlier by Kurzweil \cite{kur} who obtained bounds for the Fitting height of a soluble group $G$, and these bounds were improved to linear ones by  Meixner \cite{mei}. If it is only known that $\f$ induces a $p^n$-splitting automorphism on a $\f$-invariant Sylow $p$-subgroup of $G$, then there is already a bound in terms of $n$ for the $p$-length of $G$: for $p\ne 2$ such a bound was obtained by Wilson~\cite{wil}, and for all primes $p$ in \cite{khu-shu132} even under a weaker assumption.

In this paper we consider the natural generalization of Thompson's problem for a
$p$-soluble
group $G$  admitting an automorphism $\f$ of order $p^n$ in which the condition that  $\f$  acts  fixed-point-freely on $\f$-invariant $p'$-sections  is replaced by that $\f$ acts almost fixed-point-freely on these sections.
It is actually sufficient to impose the restriction on the number of fixed points of $\f$ only on elementary abelian  $\f$-invariant $p'$-sections.

\begin{theorem}\label{t1}
If a finite $p$-soluble group $G$ admits an automorphism $\varphi$ of order $p^n$ such that $\f $ has at most $m$ fixed points on every $\f$-invariant elementary abelian $p'$-section of $G$, then the $p$-length of $G$ is bounded above in terms of $p^n$ and $m$.
\end{theorem}

It would be interesting to obtain a bound of the $p$-length in terms of $n$ (or at least  in terms of $p^n$) for  some subgroup of index bounded in terms of $p^n$ and $m$.

\begin{remark}\label{r1}
There is a certain similarity with the situation for a $p^n$-splitting automorphism described above. Namely, if, for  a $p$-soluble group $G$ with an automorphism $\f$ of order $p^n$,  instead of a restriction on the number of fixed points on $p'$-sections, we have a restriction $|C_P(\f )|=p^m$ on the number of fixed points of $\f$ in a $\f$-invariant Sylow $p$-subgroup $P$, then we also obtain a bound for the $p$-length of $G$. Indeed, then the derived length of $P$ is bounded in terms of $p$, $n$, and $m$ by Shalev's theorem \cite{sha93}, so the bound for the $p$-length immediately follows from the Hall--Higman theorems \cite{ha-hi} for $p\ne 2$, and the theorems of Hoare \cite{hoa}, Berger and Gross \cite{ber-gro}, and Bryukhanova \cite{bry}. Moreover, by \cite{khu93} the group $P$ even has a (normal) subgroup of index bounded  in terms of $p$, $n$, and $m$ that has $p^n$-bounded derived length. Therefore  by the Hall--Higman--Hartley Theorem~\ref{t-hhh} (see below) there is a characteristic subgroup $H$ of $G$ such that the $p$-length of $H$ is $p^n$-bounded and a Sylow $p$-subgroup of  the quotient $G/H$ has order bounded in terms of $p$, $n$, and $m$.
 \end{remark}

For soluble groups, Theorem~\ref{t1} can be combined with known results to give a bound for the Fitting height.

\begin{corollary}\label{c1}
If a finite soluble group $G$ admits an automorphism $\varphi$ of order $p^n$ such that $\f $ has at most $m$ fixed points on every $\f$-invariant elementary abelian $p'$-section of $G$, then the Fitting height of $G$ is bounded above in terms of $p^n$ and $m$.
\end{corollary}

The technique used in the proof of Theorem~\ref{t1} is also applied in the proof of the soluble case of the following theorem on almost fixed-point-free automorphism of biprimary order; the reduction to the soluble case is given by Hartley' theorem \cite{har92} (based on the classification of  finite simple groups).

\begin{theorem}\label{t2}
If a finite group  $G$ admits an automorphism $\varphi$ of order $p^aq^b$ for some primes $p,q$ and nonnegative integers $a,b$, then $G$ has a soluble subgroup whose index and  Fitting height are bounded above in terms of $p^aq^b$ and $|C_G(\varphi )|$.
\end{theorem}

Standard inverse limit arguments yield the following corollary for locally finite groups.

\begin{corollary}\label{c2}
If a locally finite group $G$ contains an element $g$ of order $p^aq^b$ for some primes $p,q$ and nonnegative integers $a,b$ with finite centralizer  $C_G(g)$, then  $G$ has a subgroup of finite index that has a finite normal series with locally nilpotent factors.
\end{corollary}

Another corollary is of more technical nature but it may be useful in further studies.

\begin{corollary}\label{c3}
If a finite group $G$ admits  an automorphism $\varphi$ such that there are at most two primes dividing both $|\varphi  |$ and $|G|$, then
$G$ has a soluble subgroup whose index and  Fitting height are bounded above in terms of $|\f |$ and $|C_G(\varphi )|$.
\end{corollary}

\begin{remark}\label{r2}
After this paper was prepared for publication, the author became aware of an unpublished manuscript of Brian Hartley, which contains the result of Theorem~\ref{t2}; the author together with A.~Borovik and P.~Shumyatsky published this manuscript as \cite{hartley} on the web-site of the University of Manchester.
\end{remark}

\section{Preliminaries}

Induced automorphisms of invariant sections are denoted by the same letters.
The following lemma is well known.

\begin{lemma}\label{l1}
If $\varphi$ is an automorphism of a finite group $G$ and $N$ is a normal $\varphi$-invariant subgroup, then $|C_{G/N}(\varphi )|\leq |C_G(\varphi )|$.
\end{lemma}

The next lemma is also a well-known consequence of considering the Jordan normal form of a linear transformation of order $p^k$ in characteristic $p$.
%and the Hall--Gorchakov--Merzlyakov lemma.

\begin{lemma}\label{l2}
If an elementary abelian $p$-group $P$ admits an automorphism $\varphi$ of order $p^k$ such that $|C_{P}(\varphi )|=p^m$, then the rank of $P$ is bounded in terms of $p^k$ and $m$.
\end{lemma}

We shall use the following consequence of the Hall--Higman--type theorems in Hartley's paper \cite{har80}.

\begin{theorem}[Hall--Higman--Hartley]\label{t-hhh} Let $P$ be a Sylow $p$-subgroup of a $p$-soluble group $G$. If  $R$ is a normal subgroup of $P$ and the derived length of $R$ is $d$, then $R\leq O_{p',p,p',\ldots ,p',p}(G)$, where  $p$ occurs  on the right-hand side $d$ times if $p>3$, \  $2d$ times if $p=3$,  and $3d$ times if $p=2$.
\end{theorem}

\begin{proof} As a refinement of some of the Hall--Higman theorems  \cite{ha-hi}, Hartley \cite{har80} proved that if $A$ is an abelian normal subgroup of a Sylow $p$-subgroup of $G$, then
$$
A\leq O_{p',p}(G)\qquad \text{if}\quad p>3,
$$
$$
A\leq O_{3',3,3',3}(G)\qquad \text{if}\quad p=3,
$$
and
$$
A\leq O_{2',2,2',2,2'2}(G)\qquad \text{if} \quad p=2.
$$
 The result follows   from these inclusions for $A=R^{(d-1)}$ by a straightforward induction on the derived length~$d$.
\end{proof}

We now recall some definitions and notation from representation theory. If $V$ is a $kG$-module for a field $k$ and a group $G$, we use the right operator  notation $vg$ for $v\in V$ and $g\in G$. We use  the
centralizer notation for fixed points, like $C_V(g )=\{v\in V\mid vg=v\}$.  We also use the commutator notation $[v,g]=-v+vg$ for $v\in V$ and $g\in  G$. The commutator subspaces are defined accordingly:  if $B\leq  G$, then  $[V,B]$ is the span of all commutators $[v,b]$, where $v\in V$ and $b\in B$. The subspace $[V,B]$ coincides with the commutator subgroup $[V,B]$ in the natural semidirect product $VG$ when $V$ is regarded as the additive group acted upon by $G$. In particular, $[V,B]$ is $B$-invariant, and thus can be regarded as a $kB$-submodule.

For a group $G$ and a field $k$, a~free $kG$-module of rank $n$ is a direct sum of $n$ copies of the
group algebra $kG$ each of which is regarded as a vector
space over $k$ of dimension $|G|$ with a basis $\{b_g\mid g\in
G\}$ labelled by elements of $G$ on which $G$ acts in a regular permutation
representation: $b_gh=b_{gh}$. In other words, a free $kG$-module $V=\bigoplus_{g\in G} V_g$ is a direct sum of subspaces that are regularly permuted by $G$ so that $V_gh=V_{gh}$.

The following lemma is known in the literature (see, for example, \cite[Lemma~4.5]{ha-tu}), but we give a proof for completeness.

\bl \label{l-free} Suppose that an abelian $p$-group $M$ is acted upon by a cyclic group $\langle\f\rangle$ of order $p^n$ and $V$ is a $kM\langle\varphi \rangle$-module for a field $k$  of characteristic different from $p$. If the subgroup  $[M,\varphi ^{p^{n-1}}]$ acts non-trivially on $V$, then the subspace $[V, [M,\varphi ^{p^{n-1}}]]$ is a free $k\langle\varphi \rangle$-module.
\el

Here, of course, $\varphi ^{p^{n-1}}=\f $ if $n=1$.

\bp
The subspace $[V, [M,\varphi ^{p^{n-1}}]]$  is clearly $M\langle\varphi\rangle $-invariant, so is an $kM\langle\varphi \rangle$-module.
We extend the ground field to its algebraic closure $\bar k$ and denote by $W=V\otimes _{k}\bar k$ the resulting $\bar k M\langle\varphi\rangle$-module.  Then $[W, [M,\varphi ^{p^{n-1}}]]$ is a $\bar k M\langle\varphi \rangle$-module obtained from $[V, [M,\varphi ^{p^{n-1}}]]$ by the field extension.

Since the characteristic of the ground field is coprime to $|M\langle\f\rangle |$,  by Maschke's theorem
$$
W=C_W([M,\varphi ^{p^{n-1}}])\oplus  [W, [M,\varphi ^{p^{n-1}}]]
$$
is a  completely reducible $\bar k M\langle\varphi \rangle$-module.   Let $U$ be an irreducible  $\bar k M\langle\varphi\rangle$-submodule of $[W, [M,\varphi ^{p^{n-1}}]]$ on which $[M,\varphi ^{p^{n-1}}]$ acts non-trivially.

By Clifford's theorem,  $U=U_1\oplus\dots\oplus U_m$ decomposes into homogeneous $\bar kM$-submodules $U_i$ (Wedderburn components). The group $\langle\varphi\rangle$ transitively permutes the $U_i$. If the kernel of this permutational action was non-trivial, then $\varphi ^{p^{n-1}}$ would stabilize all the $U_i$. But the abelian group $M$ acts by scalar transformations on each homogeneous component $U_i$. Hence $[M,\varphi ^{p^{n-1}}]$ would act trivially on each $U_i$ and therefore on $U$, contrary to our assumption. Thus, $U$ is a  free $\bar k\langle\varphi \rangle$-module.

Since $[W, [M,\varphi ^{p^{n-1}}]]$ is the direct sum of such $U$, we obtain that  $[W, [M,\varphi ^{p^{n-1}}]]$ is also a free $\bar k\langle\varphi \rangle$-module. Then $[V, [M,\varphi ^{p^{n-1}}]]$ is a free $k\langle\varphi \rangle$-module. Indeed,  by the Deuring--Noether theorem \cite[Theorem~29.7]{cur-rai} two representations over a
smaller field are equivalent if they are equivalent over a larger field. Being a free $\bar k\langle\f \rangle$-module,   or a free $k\langle\f\rangle $-module, means having a basis, as of a vector space over the corresponding field, elements of which are permuted by $\f$ so that all orbits are regular. In such a basis $\langle\f\rangle$ is represented by the corresponding permutational matrices, all of which are defined  over~$k$.
\ep

\section{Automorphism of order $p^n$}\label{s-pn}

First we state separately the following proposition, which will  also be used in the next section in a different situation.

  \bpr\label{p1} Suppose that a cyclic group $\langle\f\rangle$ of order $p^n$ acts by automorphisms on a finite $p$-group $P$, and $V$ is a faithful  $\F _qP\langle\f\rangle$-module, where $\F _q$ is a prime field of order $q\ne p$. Then the derived length of $[P,\f^{p^{n-1}}]$ is bounded in terms of $|C_V(\f )|$ and $p^n$.
 \epr

 \bp  Let $M$ be a maximal abelian normal subgroup of the semidirect product $P\langle\varphi\rangle$. If $[M,\varphi ^{p^{n-1}}]\ne 1$, then by Lemma~\ref{l-free},  $[V, [M,\varphi ^{p^{n-1}}]]$  is a free $\F _q\langle\varphi \rangle$-module. Obviously, in a free $\F _q\langle\varphi \rangle$-module the fixed points of $\f$ are exactly the ``diagonal'' elements. Hence the order of $[V, [M,\varphi ^{p^{n-1}}]]$ is equal to
 $$
 |C_{[V, [M,\varphi ^{p^{n-1}}]]}(\varphi )|^{|\f |}= |C_{[V, [M,\varphi ^{p^{n-1}}]]}(\varphi )|^{p^n}
 $$
 and therefore is bounded  in terms of $|C_V(\f )|$ and $p^n$. The group $[M,\varphi ^{p^{n-1}}]$ acts faithfully on $V$; therefore by Maschke's theorem it also acts faithfully on $[V, [M,\varphi ^{p^{n-1}}]]$. Hence the order of $[M,\varphi ^{p^{n-1}}]$ is bounded in terms of $|C_V(\f )|$ and $p^n$. The same of course holds if $[M,\varphi ^{p^{n-1}}]=1$.

It follows that the index $|M:C_M(\varphi ^{p^{n-1}})|$ is bounded in terms of $|C_V(\f )|$ and $p^n$, since this index is equal to the number of different commutators $[m,\varphi ^{p^{n-1}}]$ for $m\in M$.

Consider a central series of $P\langle\varphi \rangle$ connecting $1$ and $M$.
Since  $|M:C_M(\varphi ^{p^{n-1}})|$ is bounded in terms of $|C_V(\f )|$ and $p^n$, the number of factors of this series that are not covered by  $C_M(\varphi ^{p^{n-1}})$ is bounded in terms of $|C_V(\f )|$ and $p^n$. Therefore there is a normal series of bounded length connecting $1$ and $M$ each factor of which is either central in $P\langle\varphi \rangle$ or is covered by $C_M(\varphi ^{p^{n-1}})$. Obviously, then $\varphi ^{p^{n-1}}$ acts trivially on each factor of this series, and therefore so does $[P, \varphi ^{p^{n-1}}]$. By Kaluzhnin's theorem, the automorphism group induced by the action of  $[P, \varphi ^{p^{a-1}}]$ on $M$ is nilpotent of bounded class. Since $M$ contains its centralizer in $P\langle\varphi\rangle$, it follows that $[P, \varphi ^{p^{n-1}}]$ is soluble of bounded derived length, since by the above  $\gamma _s([P, \varphi ^{p^{n-1}}])\leq  M\cap [P, \varphi ^{p^{n-1}}]$
for some  number $s$ bounded in terms of $|C_V(\f )|$ and $p^n$.
\ep

\begin{proof}[Proof of Theorem~\ref{t1}] Recall that $G$ is a finite $p$-soluble group admitting an automorphism $\varphi$ of order $p^n$ such that $\f$ has at most $m$ fixed points in every $\f$-invariant elementary abelian $p'$-section of $G$. We need to bound the $p$-length of $G$ in terms of $p^n$ and $m$. Henceforth in this section, saying for brevity that a certain parameter is simply ``bounded'' we mean that this parameter is bounded above in terms of $p^n$ and $m$.

We use induction on $n$. It is convenient to consider the case of $n=0$ as the basis of induction, when $|\f|=p^0=1$, that is, $\f$ acts trivially on $G$.  Then the hypothesis means that every elementary abelian $p'$-section of $G$ has bounded order.  We claim that the nilpotency class of a Sylow $p$-subgroup $P$ of $\hat G=G/O_{p}(G)$ is bounded.   Indeed, since the order of $P$ is coprime to  $|O_{p'}(\hat G)|$, for every prime $q$ dividing $|O_{p'}(\hat G)|$ there is a $P$-invariant Sylow $q$-subgroup $Q$ of $O_{p'}(\hat G)$. The quotient $P/C_P(Q)$ acts faithfully on the Frattini quotient $Q/\Phi (Q)$, which has order at most $m$ by the assumption. Hence $P/C_P(Q)$ has bounded order and therefore bounded nilpotency class. Since $P$ acts faithfully on  $O_{p'}(\hat G)$, we have $\bigcap C_P(Q_i)=1$, where $Q_i$ runs over all  $P$-invariant Sylow subgroups of $O_{p'}(\hat G)$. Hence $P$ is a subdirect product of groups of bounded nilpotency class and therefore has bounded nilpotency class itself.
We now obtain that the $p$-length of $\hat G=G/O_{p}(G)$ is bounded by the Hall--Higman theorem \cite{ha-hi}.  As a result, the $p$-length of $G$ is bounded.

 From now on we assume that $n\geq 1 $.

Let $\hat G=G/O_p(G)$. Consider a Sylow $p$-subgroup of the semidirect product $\hat G\langle\f\rangle$ containing $\langle\f\rangle$ and let $P$ be its intersection with $\hat G$, so that $P$ is a $\f$-invariant Sylow $p$-subgroup of $\hat G$. Since the order of the $p$-group $P\langle\f\rangle$ is coprime to  $|O_{p'}(\hat G)|$, for every prime $q$ dividing $|O_{p'}(\hat G)|$ there is a $P\langle\f\rangle$-invariant Sylow $q$-subgroup $Q$ of $O_{p'}(\hat G)$.

The quotient $\bar P=P/C_P(Q)$ acts faithfully on the Frattini quotient $V=Q/\Phi (Q)$, which we regard as an $\F _qP\langle\f\rangle$-module. By hypothesis, $|C_V(\f )|\leq m$, so by Proposition~\ref{p1}  the derived length of $[\bar P,\f^{p^{n-1}}]$ is bounded. In other words, $[P,\f^{p^{n-1}}]^{(s)}\leq C_P(Q)$ for some bounded number $s$. Since $P$ acts faithfully on  $O_{p'}(\hat G)$, we have $\bigcap C_P(Q_i)=1$, where $Q_i$ runs over all  $P\langle\f\rangle$-invariant Sylow subgroups of $O_{p'}(\hat G)$. Hence,  $[P,\f^{p^{n-1}}]^{(s)}= 1$.

By the Hall--Higman--Hartley Theorem \ref{t-hhh}  we now obtain that the normal subgroup  $[P, \varphi ^{p^{n-1}}]$ of the Sylow $p$-subgroup $P$ is contained in $H=O_{p',p,p',\ldots ,p',p}(\hat G)$, where  $p$ occurs boundedly many times.

Consider the action of $\varphi$ on the quotient $\tilde G=\hat  G/H$. Since $[P, \varphi ^{p^{n-1}}]\leq H$, it follows that $\varphi ^{p^{n-1}}$ acts trivially on the image of $P$, which is a Sylow $p$-subgroup of $\tilde G$. In particular, $\varphi ^{p^{n-1}}$ acts trivially on $O_{p',p}(\tilde G)/O_{p'}(\tilde G)$, and therefore so does $[\tilde G, \varphi ^{p^{n-1}}]$. Since $O_{p',p}(\tilde G)/O_{p'}(\tilde G)$ contains its centralizer in $\tilde G/O_{p'}(\tilde G)$, we obtain that  $[\tilde G, \varphi ^{p^{n-1}}]\leq O_{p',p}(\tilde G)$. In other words,  $\varphi ^{p^{n-1}}$ acts trivially on  the quotient $\tilde G/ O_{p',p}(\tilde G)$. Therefore the order of the automorphism induced by $\varphi$ on $\tilde G/ O_{p',p}(\tilde G)$ is at most $p^{n-1}$. By the induction hypothesis the $p$-length of this quotient is bounded. Then the $p$-length of $G/O_{p,p'}(G)$ is bounded, and therefore the $p$-length of $G$ is bounded, as required. \ep

\bp[Proof of Corollary~\ref{c1}] Here, $G$ is a finite soluble group admitting an automorphism $\varphi$ of order $p^n$ such that $\f$ has at most $m$ fixed points in every $\f$-invariant elementary abelian $p'$-section of $G$. By Theorem~\ref{t1} the $p$-length of $G$ is bounded. It remains to obtain a bound for the Fitting height of every $p'$-factor $T$  of  the upper $p$-series consisting of the subgroups $O_{p',p,p',p,\dots}$. %!
It is known that the rank of a finite group is bounded in terms of the ranks of its elementary abelian sections. Here, by definition, the rank of a group is the minimum number $r$ such that every subgroup can be generated by $r$ elements. Of course every elementary abelian section of $C_T(\f )$ is a $\f$-invariant $p'$-section of $G$ and therefore has bounded order by hypothesis. It is also known that the Fitting height of a soluble finite group is bounded in terms of its rank.  Thus  $C_T(\f )$ has bounded Fitting height and therefore so does $G$ by Thompson's theorem \cite{tho64}.
\ep 

\begin{remark} %!
If we could obtain in Theorem~\ref{t1} a ``strong'' bound for the $p$-length,  in terms of $\alpha(\langle\f\rangle)$ only, for a subgroup of bounded index, then a similar strong bound could be obtained in Corollary~\ref{c1} for the Fitting height of a subgroup of bounded index. This would follow from a rank analogue of the Hartley--Isaacs theorem proved in  \cite{khu-maz06}, which states that if a finite soluble group $K$ admits  a soluble group of automorphisms $L$ of coprime order, then $K$ has a normal subgroup $N$ of Fitting height at most $5 (4^{\alpha (L)}-1)/3$ such that the order of $K/N$ is bounded in terms of $|L|$ and the rank of $C_K(L)$.
\end{remark} 

\section{\boldmath Automorphism of order $p^aq^b$}

\begin{proof}[Proof of Theorem~\ref{t2}] Recall that $G$ is a finite group admitting an automorphism $\varphi$ of order $p^aq^b$. By Hartley's theorem \cite{har92} (based on the classification of  finite simple groups), $G$ has a soluble subgroup of index bounded  in terms of $p^aq^b$ and $|C_G(\varphi )|$. Therefore we can assume from the outset that $G$ is soluble, so that we need to bound the Fitting height of $G$ in terms of $p^aq^b$ and $|C_G(\varphi )|$. Throughout this section we say for brevity that a certain parameter is ``bounded'' meaning that this parameter  is bounded above in terms of $p^aq^b$ and $|C_G(\varphi )|$. We use without special references the fact that  the number of fixed points of $\varphi$ in every $\varphi$-invariant section of $G$ is at most $|C_G(\varphi )|$ by Lemma~\ref{l1}.

We use induction on $a+b$. As a basis of induction we consider the case when either $a=0$ or $b=0$. Then $|\varphi |$ is a prime-power, and by the Hartley--Turau theorem \cite{ha-tu} the group $G$ has a subgroup of bounded index that has Fitting height at most $\alpha (\varphi )$.   (Actually, for our `weak' bound a simpler argument would suffice: if,  say,  $|\varphi |=p^a$, then the rank of the Frattini quotient of $O_{p',p}(G)/O_{p'}(G)$ is bounded by Lemma~\ref{l2}, which implies a bound for the Fitting height of $G/O_{p'}(G)$, and the Fitting height of $O_{p'}(G)$ is bounded in terms of $a$ by Thompson's theorem \cite{tho64}.) Moreover, the following proposition holds, which apparently was noted by Hartley but may have remained unpublished. We state this proposition in a more general form, without assuming that the automorphism has biprimary order.

\begin{proposition}
\label{p}
If  a finite soluble group $G$ admits an automorphism $\psi$ such that there is at most one prime dividing both $|\psi |$ and $|G|$, then the Fitting height of $G$ is bounded above in terms of $|\psi |$ and $|C_G(\psi )|$.
\end{proposition}

\begin{proof}
If $(|\psi |, |G|)=1$, then the result follows from the stronger theorem of Thompson \cite{tho64}. Now let $\langle \psi\rangle=\langle \psi _r\rangle\times \langle \psi  _{r'}\rangle$, where $\langle \psi _r\rangle$ is the Sylow $r$-subgroup of $\langle \psi \rangle$ and $r$ is the only common prime divisor of $|G|$ and $|\psi |$. The centralizer $C_{G}(\psi  _{r'})$ admits the automorphism $\psi  _r$ of prime-power order whose centralizer $C_{C_{G}(\psi  _{r'})}(\psi  _r)$ is equal to $C_G(\psi  )$. By the Hartley--Turau theorem, the Fitting height of $C_{G}(\psi  _{r'})$ is bounded. We now apply Thompson's theorem to the automorphism $\psi  _{r'}$ of $G$ of coprime order to obtain that the Fitting height of $G$ is bounded as required.
\end{proof}

We return to the proof of Theorem~\ref{t2}. Let $a\geq 1$ and $b\geq 1$. Let $\varphi _p=\varphi ^{q^b}$ and $\varphi _q=\varphi ^{p^a}$, so that $|\varphi _p|=p^a$ and $|\varphi _q|=q^b$, while $\langle \varphi\rangle=\langle \varphi_p\rangle\times \langle \varphi _q\rangle$. The subgroup $O_{q'}(G)$ admits the automorphism $\varphi $ whose order has at most one prime divisor $p$ in common with $|O_{q'}(G)|$. By Proposition~\ref{p}
the Fitting height of $O_{q'}(G)$ is bounded.

Therefore we can assume that $O_{q'}(G)=1$. Then the quotient $\bar G= G/O_q(G)$ acts faithfully on the Frattini quotient  $V=O_q(G)/ \Phi (O_q(G))$, which we regard as an $\F _q\bar G\langle\varphi\rangle$-module.
The fixed-point subspace $C_V(\varphi _p)$ has bounded
order. This follows from Lemma~\ref{l2} applied to the action of the linear transformation $\varphi _q$ of order $q^b$ on  $C_V(\varphi _p)$, since the fixed points of  $\varphi _q$ in  $C_V(\varphi _p)$ are contained in the  fixed-point subspace $C_V(\varphi )$ of bounded order.

 Choose a Sylow $p$-subgroup of $\bar G\langle\varphi_p\rangle$ containing  $\langle\varphi_p\rangle$, and let $P$ be its intersection with $\bar G$, so that $P$ is a $\varphi _p$-invariant Sylow $p$-subgroup of $\bar G$. By Proposition~\ref{p1},
 the subgroup  $[P, \varphi _p^{p^{a-1}}]$ has bounded derived length.

Hence by the Hall--Higman--Hartley Theorem~\ref{t-hhh} the normal subgroup  $[P, \varphi _p^{p^{a-1}}]$ of the Sylow $p$-subgroup $P$ is contained in $H=O_{p',p,p',\ldots ,p',p}(\bar G)$, where  $p$ occurs boundedly many times.

Consider the action of $\varphi$ on the quotient $\tilde G=\bar G/H$. Since $[P, \varphi _p^{p^{a-1}}]\leq H$, it follows that $\varphi _p^{p^{a-1}}$ acts trivially on the image of $P$, which is a Sylow $p$-subgroup of $\tilde G$. In particular, $\varphi _p^{p^{a-1}}$ acts trivially on $O_{p',p}(\tilde G)/O_{p'}(\tilde G)$, and therefore so does $[\tilde G, \varphi _p^{p^{a-1}}]$. Since $O_{p',p}(\tilde G)/O_{p'}(\tilde G)$ contains its centralizer in $\tilde G/O_{p'}(\tilde G)$, we obtain that  $[\tilde G, \varphi _p^{p^{a-1}}]\leq O_{p',p}(\tilde G)$. In other words,  $\varphi _p^{p^{a-1}}$ acts trivially on  the quotient $\tilde G/ O_{p',p}(\tilde G)$. Therefore the order of the automorphism induced by $\varphi$ on $\tilde G/ O_{p',p}(\tilde G)$ divides $p^{a-1}q^b$. By induction, the Fitting height of this quotient is bounded.

It remains to obtain a bound for the Fitting height of each of the boundedly many $\varphi$-invariant normal  $p'$-sections that appear in the upper $p$-series of the groups
$H$ and $O_{p',p}(\tilde G)$. Such a bound follows from Proposition~\ref{p}.
\end{proof}

\bp[Proof of Corollary~\ref{c2}] This corollary for locally finite groups follows from Theorem~\ref{t2} by the standard inverse limit argument.
\ep

\bp[Proof of Corollary~\ref{c3}] Here, a finite  group $G$ admits  an automorphism $\varphi$ such that there are at most two primes  dividing both $|\varphi  |$ and $|G|$. Again, by Hartley's theorem \cite{har92} we can assume from the outset that $G$ is soluble.
If $(|\f |, |G|)$  is $1$ or a prime power,  then the result follows from Proposition~\ref{p}.  Now let $\langle \f\rangle=\langle \f  _{pq}\rangle\times \langle \psi \rangle$, where $\langle \f _{pq}\rangle$ is the Hall $\{p,q\}$-subgroup of $\langle \f \rangle$ and $p,q$ are the only common prime divisors of $|G|$ and $|\f |$. The centralizer $C_{G}(\psi  )$ admits the automorphism $\f _{pq}$ of biprimary order whose centralizer $C_{C_{G}(\psi )}(\f  _{pq})$ is equal to $C_G(\f  )$. By Theorem~\ref{t2}, the Fitting height of $C_{G}(\psi )$ is bounded in terms of $|\f _{pq}|$ and $|C_G(\f )|$. We now apply Thompson's theorem \cite{tho64} to the automorphism $\psi $ of $G$ of coprime order to obtain that the Fitting height of $G$ is bounded in terms of $|\f |$ and $|C_G(\f )|$.
\ep

\end{document}